\documentclass{article}
\usepackage{geometry}
\usepackage[T1]{fontenc}
\usepackage[english]{babel}
\usepackage{mathtools,mathrsfs}
\usepackage{amsmath,amssymb,amsthm}
\usepackage{algorithm}
\usepackage[noend]{algpseudocode}
\usepackage{listings, enumitem}
\usepackage[nohints]{minitoc}
\usepackage{graphicx}
\usepackage{subcaption}
\usepackage{dsfont}
\usepackage{tikz}
\usepackage[nomessages]{fp}
\usepackage{pgfplots}
\pgfplotsset{compat=1.11}
\pgfplotsset{
	colormap={parula}{
	rgb255=(53,42,135)
	rgb255=(15,92,221)
	rgb255=(18,125,216)
	rgb255=(7,156,207)
	rgb255=(21,177,180)
	rgb255=(89,189,140)
	rgb255=(165,190,107)
	rgb255=(225,185,82)
	rgb255=(252,206,46)
	rgb255=(249,251,14)}
}
\usepackage{hyperref}
\usepackage{diagbox}
\newcommand{\dx}{\,\mathrm{d}}
\newcommand{\tT}{\mathrm{T}}

\newcommand{\domain}{\Omega}

\DeclareMathOperator{\R}{\mathbb{R}}

\newcommand{\NN}{\mathbb{N}}

\DeclareMathOperator*{\argmin}{arg\,min} 
\DeclareMathOperator{\prox}{prox} 

\DeclareMathOperator{\TV}{TV}

\newcommand{\N}{\mathbb{N}}
\newcommand{\Id}{\mathrm{Id}}

\newcommand{\weakly}{\rightharpoonup}

\DeclareMathOperator{\BV}{BV}

\newtheorem{theorem}{Theorem}[section]

\newtheorem{remark}[theorem]{Remark}

\title{Template-Based Image Reconstruction\\ Facing Different Topologies}
\author{Sebastian Neumayer \and Antonia Topalovic}
\date{\today}

\begin{document}

\maketitle
\begin{abstract}
    The reconstruction of images from measured data is an increasing field of research.
    For highly under-determined problems, template-based image reconstruction provides a way of compensating for the lack of sufficient data.
    A caveat of this approach is that dealing with different topologies of the template and the target image is challenging.
    In this paper, we propose a LDDMM-based image-reconstruction model that resolves this issue by adding a source term.
    On the theoretical side, we show that the model satisfies all criteria for being a well-posed regularization method.
    For the implementation, we pursue a discretize-then-optimize approach involving the proximal alternating linearized minimization algorithm, which is known to converge under mild assumptions.
    Our simulations with both artificial and real data confirm the robustness of the method, and its ability to successfully deal with topology changes even if the available amount of data is very limited.
    
\end{abstract}

\section{Introduction} \label{sec:intro_3}
In medical applications such as computed tomography (CT) \cite{N2001}, images are typically observed via indirect and potentially noisy measurements.
Especially when the amount of measured data is limited, obtaining meaningful reconstructions is challenging.
This is, for instance, the case in limited-angle CT \cite{Fri2013,N2001}, where sparse data is acquired in order to minimize exposure time of organisms to X-radiation.
In such settings, it is inevitable to add a priori information about the target into the reconstruction process, e.g., in form of a template image that is somehow close to the expected reconstruction.
Template-based methods, outlined in more detail below, encode this closeness assumption directly into the reconstruction process.
Hence, any reconstruction will strongly depend on the chosen template.
If a good template is available, e.g., from an earlier observation, competing methods such as the filtered backprojection \cite{N2001} or total-variation (TV) regularization~\cite{ROF92} are outperformed by large margins \cite[Sec.~10]{CO18}.
Clearly, template-based methods can also be applied for other inverse problems such as deblurring or MRI. 
In the following, we discus \emph{indirect image matching} in more detail.
The included examples lead us to the proposed model, which is a simplified extension of the metamorphosis approach proposed in~\cite{GO18}.

\paragraph{Indirect image matching}
Let $\Omega \subset \R^d$ be a bounded domain.
Image matching refers to the task of transforming a template image $T\in L^2(\Omega)$ to match a target image $U\in L^2(\Omega)$ as closely as possible regarding some misfit measure.
For indirect image matching, the target $U$ is additionally unknown and only specified to be a solution of the \emph{image-reconstruction} problem
\begin{equation}
	K \circ U + n^{\delta}= g,
\label{eq:operatorequation}
\end{equation}
where $K\colon L^2(\Omega) \to Y$ is a (not necessarily linear) data acquisition operator (often related to a physical process) that maps into a Banach space $Y$, $g \in Y$ is the measurement, and the noise $n^{\delta}$ models measurement imperfections.
Consequently, one has to jointly solve a motion estimation problem and an image-reconstruction problem.
Various indirect image matching models have been proposed in the literature \cite{CO18,GO18,HinSzeWanSalJos12,LNOS19,NPS18,OECDRB18}.
For all models, the transformation between the template $T$ and the unknown target $U$ can be modeled with a partial differential equation (PDE).
Essentially, the template`s domain $\Omega$ is deformed by a diffeomorphism, usually resulting in transformations that appear natural to humans.
This idea originates from the flow of diffeomorphism model \cite{CRM96, DGM98, Tro95}, in which image pixel intensities are transported along trajectories determined by diffeomorphism paths.
For comprehensive overviews, we refer to \cite{MTY15,Younes2010}, and for a historic account we refer to \cite{MTY02}.
In the following, we discuss two indirect image matching models in more detail.

As the space of diffeomorphisms has no natural vector space structure, the popular class of Large Diffeomorphic Deformation Metric Mapping (LDDMM) image matching models relies on a subset of deformations generated by admissible (smooth) velocity fields $v\in \mathcal V \coloneqq L^2([0,1], V)$ via the ordinary differential equation (ODE)
\begin{equation}\label{eq:Lag}
\begin{alignedat}{2}
&\tfrac{\dx}{\dx t} \varphi(t,x) = v\bigl( t, \varphi(t,x) \bigr)   
\quad &&\text{for } (t,x) \in [0,1]\times\Omega,
\\
&\varphi(0,x)= x \quad &&\text{for } x\in \Omega.
\end{alignedat}
\end{equation}
In the definition of $\mathcal V$, $V$ is an admissible vector space that continuously embeds into $C_0^{1,\alpha}(\Omega, \R^d)$, $0<\alpha \leq 1$, namely the closure of  $C_c^{\infty}(\Omega, \R^d)$ with respect to the Hölder norm $\Vert \cdot \Vert_{C^{1,\alpha}}$.
Using the diffeomorphism $\varphi$, we implicitly define a transformation path $I \colon [0,1] \times \Omega \to \R$ starting from the template $T$ via $I(t, \varphi(t,x)) = T(x)$.
From an Eulerian perspective, we can link this path $I$ to the velocity $v$ directly \cite{CO18}, namely as the (weak) solution of the transport equation
\begin{equation}\label{eq:transport}
\begin{alignedat}{2}
&\tfrac{\partial}{\partial t} I(t,x) + v(t,x) 
\nabla_x I(t,x)  = 0 \quad &&\text{for $(t,x)\in [0,1]\times\Omega$,} \\
&I(0,x) = T(x) \quad &&\text{for $x\in\Omega$.}
\end{alignedat}
\end{equation}
We denote the set of feasible tuples $(I,v) \in L^2([0,1],L^2(\Omega)) \times \mathcal V$ solving this PDE by $\mathcal A$.
Given the data $g$ and the template $T$, a reconstruction can now be defined as $R= I(1,\cdot) = T(\varphi^{-1}(1,\cdot))$\footnote{Throughout the manuscript, the inverse of $\varphi$ is always with respect to the spatial coordinate only.}, where $I$ is a minimizer of the variational problem
\begin{align}\label{eq:LDDMM_rec}
	\min_{(I,v) \in \mathcal A} \mathcal D\bigl(K \circ I(1,\cdot),g\bigr) + \lambda \mathcal E(v).
\end{align}
Here, the data fidelity term $\mathcal D\colon Y \times Y \to \R_{\geq 0}$ quantifies the misfit of $K \circ I(1,\cdot)$ with the measurements $g$, and the regularizer $\mathcal E \colon \mathcal V \to \R_{\geq 0}$ enforces the required smoothness of $v$ \cite{CO18}.
Note that \eqref{eq:Lag} can be simplified by using linearized deformations and a Taylor expansion around some initial flow field,
leading to optical-flow-based transformation models \cite{AndSchZul15,BorItoKun03,HS81,OECDRB18}.

However, an image transformation model purely based on diffeomorphisms can yield unsatisfying reconstructions, e.g., if the images have different mass or topological properties \cite{CO18}.
To resolve this issue, \cite{GO18} proposed to replace the flow of diffeomorphism model \eqref{eq:transport} underlying \eqref{eq:LDDMM_rec} with the metamorphosis model \cite{MY2001, RY16, TY2005a, TY2005b}.
In addition to the transport of pixel intensities, metamorphosis also allows intensity variations along the trajectories based on a source term $\zeta \in L^2([0,1] \times \Omega)$.
Hence, this model can create or remove objects during the transformation process.
To put this into formulas, metamorphosis paths $I$ are solutions of the transport equation
\begin{equation}\label{eq:transport_source}
\begin{alignedat}{2}
&\tfrac{\partial}{\partial t} I(t,x) + v(t,x) 
\nabla_x I(t,x)  = \zeta(t,x) \quad &&\text{for $(t,x)\in [0,1]\times\Omega$,} \\
&I(0,x) = T(x) \quad &&\text{for $x\in\Omega$.}
\end{alignedat}
\end{equation}
We denote the set of feasible tuples $(I,v,\zeta ) \in L^2([0,1] \times \Omega) \times \mathcal V \times L^2([0,1]\times\Omega)$ by $\mathcal B$.
From the Lagrangian perspective, $I(t,x)$ can be equivalently defined using the solution $\varphi$ of \eqref{eq:Lag}via
\begin{equation}\label{eq:LagSource}
\begin{alignedat}{2}
&I\bigl(t,\varphi(t,x)\bigr) = T(x) + \int_0^t \zeta ( s, \varphi(s,x)) \dx s 
\quad &&\text{for } (t,x) \in [0,1]\times\Omega.
\end{alignedat}
\end{equation}
The relation between \eqref{eq:transport_source} and \eqref{eq:LagSource} is commonly known as the method of characteristics \cite{Evans1998}.
Given the data $g$ and the template $T$, a reconstruction can then be defined as $R= I(1,\cdot)$, where $I$ solves
\begin{align}\label{eq:Indirect_Meta}
\min_{(I,v,\zeta ) \in \mathcal B} \mathcal D\bigl(K \circ I(1,\cdot),g\bigr) + \lambda_1 \mathcal E_1(v) + \lambda_2\mathcal E_2(\zeta ).
\end{align}
Here, the additional regularizer $\mathcal E_2 \colon L^2([0,1] \times \Omega) \to \R_{\geq 0}$ enforces the necessary regularity of $\zeta$.

\paragraph{Proposed model}
If we take a closer look at \eqref{eq:Indirect_Meta}, we notice that the objective takes only $I(1,\cdot)$ into account and is blind to the path $I$ at all other times.
Further, it holds that $I(1,\cdot) = T(\varphi^{-1}(1,\cdot)) + z$, where $z \in L^2(\Omega)$ is defined via \smash{$z(\varphi(1,x))= \int_0^1 \zeta(s, \varphi(s,x)) \dx s$}.
Consequently, any reconstruction $R$ consists of a template deformation and some intensity change $z$ that depends on $\zeta$ and $\varphi$.
In \eqref{eq:Indirect_Meta}, we have a complicated regularization of $z$ in terms of both $\mathcal E_1$ and $\mathcal E_2$.
To obtain a simpler reconstruction model, we propose to replace the underlying transformation model \eqref{eq:transport_source} in \eqref{eq:Indirect_Meta} by
\begin{equation}\label{eq:Lag_simp}
\begin{alignedat}{2}
&\tfrac{\dx}{\dx t} \varphi(t,x) = v\bigl( t, \varphi(t,x) \bigr)   
\quad &&\text{for } (t,x) \in [0,1]\times\Omega,
\\
&\varphi(0,x)= x \quad &&\text{for } x\in \Omega,
\\
&I(x) = T\bigl(\varphi^{-1}(1,x)\bigr) + z(x) \quad &&\text{for } x\in \Omega.
\end{alignedat}
\end{equation}
The set of tuples $(I,v,z) \in L^2(\Omega) \times \mathcal V \times L^2(\Omega)$ satisfying \eqref{eq:Lag_simp} is denoted by $\mathcal C$.
Loosely speaking, we first deform the template $T$ based on $\varphi$.
Afterwards, we modify the pixel values with the source $z$.
Based on the simplified transformation model \eqref{eq:Lag_simp}, we propose the reconstruction model
\begin{align}\label{eq:Indirect_Meta_simp}
\min_{(R,v,z ) \in \mathcal C} \mathcal D\bigl(K \circ R,g\bigr) + \lambda_1 \mathcal E_1(v) + \lambda_2\mathcal E_2(z).
\end{align}
Instead of implicitly regularizing $z$ as in \eqref{eq:Indirect_Meta}, we directly regularize it with $\mathcal E_2 \colon L^2(\Omega) \to \R_{\geq 0}$.
To this end, we can rely on any well-established (convex) regularizer for images.
The reconstruction $R$ in \eqref{eq:Indirect_Meta_simp} is composed of two summands, where one depends on $v$ and the other on $z$.
Therefore, designing efficient numerical schemes for \eqref{eq:Lag_simp} appears simpler than for \eqref{eq:LagSource}, where we have a highly nonlinear dependence of $R$ on $\zeta$.
Additionally, the dimensionality of $z$ is lower than that of $\zeta$ as we do not have a time dependence.
Since each tuple $(v,z)$ gives rise to a unique reconstruction $R_{v,z}$, we can eliminate the set $\mathcal C$ from \eqref{eq:Indirect_Meta_simp} and end up with
\begin{align}\label{eq:Indirect_Meta_elim}
\min_{(v,z ) \in \mathcal V \times L^2(\Omega)} J_{\lambda,g} (v,z) \coloneqq \mathcal D\bigl(K \circ R_{v,z},g\bigr) + \lambda_1 \mathcal E_1(v) + \lambda_2\mathcal E_2(z),
\end{align}
where the system \eqref{eq:Lag_simp} is implicitly encoded in $R_{v,z}$.

\paragraph{Contributions}
We study the regularizing properties of the proposed model \eqref{eq:Indirect_Meta_elim}, and develop efficient numerical schemes for solving it.
To the best of our knowledge, the only other template-based reconstruction model with a source term is \cite{GO18}, but their approach involves a nonlinear coupling between the deformation and the source part.
Further, their regularizer choice is restricted to differentiable ones.
The simulations in \cite{GO18} are based on $L^2$ regularization for $\zeta$ in \eqref{eq:Indirect_Meta}, which in our experiments turned out to be unsuitable for sparse data.
Better regularizers seem to be necessary to avoid reconstruction artifacts related to the source term in both \eqref{eq:Indirect_Meta} and \eqref{eq:Indirect_Meta_elim}.
Hence, we propose to use TV regularization instead, which is known to yield good results for CT problems at reasonable computational cost.
With this choice, we obtain meaningful reconstructions even for sparse data, which is a setting that \cite{GO18} does not target.
Compared to \cite{LNOS19}, we achieve reconstructions of similar quality, but with the advantage that topology changes are possible due to $z$.
The detail of reconstructable structures that are not present in the template $T$ depends on the available amount of data $g$, i.e., we cannot reconstruct detailed structures out of nothing.
Hence, the proposed method is most useful if a good template $T$ is available.

Many algorithmic approaches for the flow of diffeomorphism model have been proposed during the last years \cite{Ash07, MilTroYou06, Hernandez2017, NHPV2011, SinHinJosFle13, VRRC12, YanKwiStyNie17}.
For linear forward operators $K$, these approaches can be often extended to the indirect setting \eqref{eq:Indirect_Meta_simp} without any complicated modifications.
As Lagrangian approaches turned out to be very efficient for indirect image matching, we decided to adapt the methods developed in \cite{LNOS19,ManRut17}, which build upon the FAIR toolbox \cite{Mod2009}.
Since the problem is non-smooth due to the TV regularization of $z$, we cannot deploy their proposed Gauss--Newton--Krylov solver, and we use the iPALM algorithm \cite{PS17} instead.
Similarly as in \cite{LNOS19}, the ODE in \eqref{eq:Lag_simp} is solved with an explicit Runge--Kutta method, which allows one for efficient algorithmic differentiation of $(v,z) \mapsto R_{v,z}$.
By construction of $R_{v,z}$, computing its gradient $\nabla R_{v,z}$ has basically the same cost as for the LDDMM-based model \cite{LNOS19}, which does not involve a source term $z$.
Further, the approach does not require the storage of multiple space-time vector fields or images at intermediate time instances, as it is often the case when directly solving the PDE \eqref{eq:transport_source} with Eulerian methods.
We want to emphasize that the proposed scheme can be implemented matrix-free, which is crucial when using dense forward operators $K$ such as the Radon transform.
Since iPALM has in general worse convergence rates than second-order methods, we combine it with a Gauss--Newton solver for $v$ as post-processing step.
This can only improve the objective function values, and in practice we also observed an improved reconstruction quality.

\paragraph{Outline}
In Section~\ref{sec:prelim}, the necessary theoretical background for the flow equation \eqref{eq:Lag} and the total variation is provided.
For the proposed model \eqref{eq:Indirect_Meta_elim}, existence of a minimizer, stability with respect to the data, and convergence for vanishing noise are established in Section~\ref{sec:proper}.
In order to approximate solutions of~\eqref{eq:Indirect_Meta_elim} numerically, we follow a \emph{discretize-then-optimize} approach that involves the iPALM algorithm as outlined in Section~\ref{sec:discrete}.
This allows one to easily exchange the regularizer for $v$ and $z$ if desired.
Our implementation builds upon the FAIR toolbox \cite{Mod2009}, which allows for a simple extension to other distances and regularizers that are already implemented as part of the toolbox.
Numerical results for the proposed model are provided in Section~\ref{sec:experiments}.
Finally, conclusions are drawn in Section~\ref{sec:conclusions}.

\section{Preliminaries}\label{sec:prelim} 
In this section, we briefly review the necessary theoretical background.
\paragraph{Diffeomorphisms}
Recall that the deformations $\varphi$ are induced by \eqref{eq:Lag}.
For our theoretical investigations, it is useful to consider different initial times $s \in [0,1]$, i.e., the modified equation
\begin{equation}\label{eq:Lag_mod}
\begin{alignedat}{2}
&\tfrac{\dx}{\dx t} \varphi_{s,v}(t,x) = v\bigl( t, \varphi_{s,v}(t,x) \bigr)   
\quad &&\text{for } (t,x) \in [0,1]\times\Omega,
\\
&\varphi_{s,v}(s,x)= x \quad &&\text{for } x\in \Omega. 
\end{alignedat}
\end{equation}
Here, the subscripts express the dependence on the initial time $s$ and the velocity field $v$.
This generalization enables us to move back and forth on trajectories $\varphi_{s,v}(\cdot,x)$ starting from arbitrary time instants $s$, allowing us to rely on a unified theoretical result.
The following theorem is a reformulation of \cite[Thms.~1 and 9]{TY2005a} and characterizes the solutions of \eqref{eq:Lag_mod}.
\begin{theorem}\label{thm:DiffeoVelo}
	Let $v \in L^2([0, T], \mathcal V)$, where $\mathcal V$ is continuously embedded into $C_0^{1,\alpha}(\Omega, \R^d)$ for some $0<\alpha \leq 1$.
	Given $s \in [0,T]$, there exists a unique global solution $\varphi_{s,v} \in C([0,T],C^1(\overline{\Omega}, \R^d))$  of \eqref{eq:Lag_mod}.
	Further, the solution operator \smash{$\Phi_s \colon L^2([0,T],\mathcal V) \to C([0,T] \times \overline \Omega, \R^d)$} assigning a flow $\varphi_{s,v}$ to a velocity field $v$ is continuous with respect to the weak topology in $L^2([0,T],\mathcal V)$.
\end{theorem}
As $\varphi_{t,v}(0,\varphi_{0,v}(t,x)) = x$, we directly get that $\varphi_{0,v}(t,\cdot)$ is a diffeomorphism for every $t \in [0,1]$.
Now, let us have a closer look at the solutions of \eqref{eq:Lag_simp}.
Since $\varphi_{0,v}^{-1}(1,\cdot) = \varphi_{1,v}(0,\cdot)$, we conclude by Theorem \ref{thm:DiffeoVelo} that $v_i \weakly v$ in $L^2([0, T],\mathcal V)$ implies \smash{$\varphi_{0,v_i}^{-1}(1,\cdot) \to \varphi_{0,v}^{-1}(1,\cdot) \in C(\overline \Omega, \R^d)$}.
Further, \cite[Thm.~3.1.10]{Effland17} implies that \smash{$\{\varphi_{0,v_i}^{-1}(1,\cdot)\}_{i\in \NN}$} is uniformly bounded in $C^{1,\alpha}(\overline{\domain}, \R^d)$.
Hence, \cite[Cor.~3]{NPS17} implies that \smash{$T \circ \varphi_{0,v_i}^{-1}(1,\cdot) \to T\circ \varphi_{0,v}^{-1}(1,\cdot)$} in $L^2(\domain)$.
If further $z_i \weakly z$ in $L^2(\Omega)$, this directly implies $R_{v_i,z_i} \weakly R_{v,z}$ in $L^2(\Omega)$.

\paragraph{Total variation}
The total variation \cite{ROF92} is a popular regularizer in imaging as it tends to preserve edges and sharp structures, which is in contrast to other techniques such as linear smoothing or Tikhonov regularization \cite{BL18}.
For any $f \in L^1(\Omega)$, the \emph{distributional gradient} 
$\nabla f \in C^1_c(\Omega, \R^d)^*$ is given by
\begin{equation}\label{eq:Gradient}
  \nabla f(\psi) \coloneqq - \int_\Omega f \mathrm{div}\, \psi \dx x \qquad \forall \psi \in C^1_c(\Omega, \R^d).
\end{equation}
Based on this, the \emph{total variation} of $f \in L^1(\Omega,\mathbb{R^d})$ is introduced as
\begin{equation}
    \TV(f) \coloneqq \sup \bigl\{\nabla f(\psi) \;:\; \psi \in C^1_c(\Omega, \R^d),\, \|\psi\|_\infty \le 1 \bigr\},
\end{equation}
and the \emph{functions of bounded variation} are defined as
\begin{equation}   \label{eq:BV}
\BV(\Omega) \coloneqq \bigl\{ f \in L^1(\Omega) \;:\; \TV(f) < \infty \bigr\}.
\end{equation}
Equipped with the norm
$\|f\|_{BV} \coloneqq \|f\|_{L^1(\Omega)} + \TV(f)$, $\BV(\Omega)$ becomes a Banach space.
A function $\mu\colon\mathcal B(\Omega) \to \R^d$ on the Borel $\sigma$-algebra $\mathcal B(\Omega)$ is called a 
\emph{vector-valued Radon measure} if every coordinate function $\mu_{i}\colon\mathcal B(\Omega) \to \mathbb R$ is a Radon measure. 
We denote by $\mathcal M(\Omega, \R^d)$  the \emph{space of vector-valued finite Radon measures}. 
Due to the Riesz--Markov--Kakutani representation theorem, it holds $C_0(\Omega, \R^d)^* \cong \mathcal M(\Omega, \R^d)$, where $C_0(\Omega, \R^d)$ denotes the continuous functions vanishing at infinity.
Hence, we can equip $\mathcal M(\Omega, \R^d)$ with the associated \emph{weak* convergence}.
For $f \in \mathrm BV(\Omega, \R^d)$ it holds
\begin{equation}
    |\nabla f(\psi)| \le \TV(f) \|\psi\|_\infty \qquad \forall \psi \in C^1_c(\Omega, \R^d),
\end{equation}
and since $C_c^1(\Omega, \R^d)$ 
is dense and continuously embedded in $C_0(\Omega, \R^d)$, 
the gradient $\nabla f$ can be uniquely extended to a continuous linear functional on $C_0(\Omega, \R^d)$ using the Hahn--Banach theorem.
Therefore, we can associate a unique measure $Df \in \mathcal M(\Omega, \R^d)$ to $\nabla f$ such that 
\begin{equation}   \label{eq:TVmeasure}
  \nabla f(\psi) = 
  \sum_{i =1}^d \int_\Omega \psi_i \dx Df_{i} \qquad \forall \psi \in C^1_c(\Omega, \R^d).
\end{equation}
Hence, $\BV(\Omega)$ consists of those functions $f \in L^1(\Omega)$ having a distributional gradient that is a finite Radon measures.
Note that the domain of $\TV$ can be extended to $L^p(\Omega)$, $p \in [1, \infty)$, via
\begin{equation} \label{BV}
	{\TV} (f) \coloneqq 
	\left\{
	\begin{array}{ll}
	\TV(f)& \mathrm{for} \; f \in BV(\Omega),\\
	+\infty & \mathrm{for} \; f \in L^p(\Omega)\backslash BV(\Omega).
	\end{array}
	\right.
\end{equation}
It is well-known that this extension is proper, convex and (weakly) lower semi-continuous (lsc) \cite[Lem.~6.105]{BL18}.
Furthermore, we have for $p \leq \frac{d}{d-1}$ the continuous embedding $\BV(\Omega) \hookrightarrow L^p(\Omega)$ as well as the \textit{Poincaré--Wirtinger inequality}
$\| P_0 f \|_p \leq \TV(f)$, where $P_0 f \coloneqq f - \vert \Omega \vert^{-1}\int_\Omega f \dx x$. 
More precisely, $P_0$ is the projection onto the complement of the subspace $\Pi_0$ of the constant functions.
The projection onto $\Pi_0$ itself is denoted by $Q_0$.
Consequently, $\TV(f)$ is coercive in the sense that $\Vert P_0 f_n \Vert_p \to \infty$ implies $\TV(f_n) \to \infty$.
This can be used to prove coercivity of functionals involving $\TV$ regularization, provided that the remaining terms are coercive with respect to $\Vert Q_0 f_n \Vert_p$.

\section{Regularizing Properties}\label{sec:proper}
Here, we study the regularizing properties of the reconstruction model \eqref{eq:Indirect_Meta_elim} following the considerations in \cite[Sec.~3]{LNOS19} and making the necessary modifications due to additional source term $z$.
\paragraph{General case}
Throughout this section, we assume that $V$ fulfills the regularity requirements from Theorem \ref{thm:DiffeoVelo}, i.e., $V \hookrightarrow C_0^{1,\alpha}(\Omega,\R^d)$ for some $0<\alpha \leq 1$.
Regarding the data fidelity term $\mathcal D$, the forward operator $K$, and the regularizers $\mathcal E_1$, $\mathcal E_2$, we make the following assumptions:
\begin{enumerate}
\item\label{cond:Cont} The operator $K\colon L^2(\Omega) \to Y$ is continuous and weak-weak-continuous, i.e., $x_n \weakly x$ in $L^2(\Omega)$ implies $K(x_n) \weakly K(x)$ in $Y$.
\item The functional $\mathcal D(\cdot, g)$ is weakly lsc for all $g\in Y$ and $D(f,\cdot)$ is continuous for all $f \in Y$.
\item\label{cond:Zero} If $\mathcal D(f,g)=0$, then it holds that $f=g$.
\item\label{cond:InvTria} There exists $C>0$ such that it holds $\vert \mathcal D(f,h)-\mathcal D(f,g) \vert \leq   C \mathcal D(g,h)$ for all $f,g,h \in Y$.
\item\label{cond:coercive} For fixed $g \in Y$, any bounded sequence $\{f_n\}_{n \in \N} \subset L^2(\Omega)$ and any sequence $\{z_n\}_{n \in \N} \subset L^2(\Omega)$ with $\Vert z_n\Vert_2 \to \infty$ it holds that $\mathcal D(K(f_n + z_n),g) + \mathcal E_2(z_n)\to \infty$.
Further, for $\{v_n\}_{n \in \N} \subset \mathcal V$ with $\Vert v_n \Vert_{\mathcal V} \to \infty$ it holds that $\mathcal E_1(v_n) \to \infty$.
This can be interpreted as coercivity in $v$ and $z$.
\item\label{cond:reg} The regularizers $\mathcal E_1$ and $\mathcal E_2$ are weakly lsc.
\end{enumerate}
\begin{remark}
These conditions readily imply that if $\{f_n\}_{n \in \N}, \{g_n\}_{n \in \N}$ are sequences in $Y$ with $f_n \rightharpoonup f$ and $g_n \to g$, then $\liminf_{n \to \infty} \mathcal D(f_n,g) = \liminf_{n \to \infty} \mathcal D(f_n,g_n)$.
Further, it is easy to verify that Conditions \ref{cond:Cont}-\ref{cond:InvTria} are fulfilled if $\mathcal D$ is a metric and $K$ a bounded linear operator.
\end{remark}
First, we prove existence of a minimizer for \eqref{eq:Indirect_Meta_elim}.
\begin{theorem}[Existence of minimizers]\label{thm:exist}
For any $\lambda_1, \lambda_2>0$ and $g \in Y$, the problem \eqref{eq:Indirect_Meta_elim} has a minimizer.
\end{theorem}

\begin{proof}
Let $\{v_n, z_n\}_{n \in \N} \subset \mathcal V \times L^2(\Omega)$ be a minimizing sequence for \eqref{eq:Indirect_Meta_elim}.
As $\mathcal E_1(v_n) \leq J(v_0,z_0)/\lambda_1$ holds for all $n \in \N$, the sequence $\{v_n\}_{n \in \N}$ is bounded in $\mathcal V$ by Condition~\ref{cond:coercive}.
Hence, there exists a subsequence, also denoted with $\{v_n\}_{n \in \N}$, such that $v_n \weakly v^*$ for some $v^*\in \mathcal V$.
By Theorem~\ref{thm:DiffeoVelo} and the discussion thereafter, the sequence \smash{$f_n \coloneqq T \circ \varphi_{0,v_n}^{-1}(1,\cdot)$} converges strongly in $L^2(\Omega)$ to $f \coloneqq T \circ \varphi_{0,v^*}^{-1}(1,\cdot)$ and is thus bounded.
By Condition~\ref{cond:coercive}, we conclude that also $\{z_n\}_{n \in \N}$ is bounded.
Therefore, we can extract a weakly convergent subsequence, also denoted with $\{z_n\}_{n \in \N}$, with limit $z^* \in L^2(\Omega)$.
As $K$ is weak-weak-continuous (Condition~\ref{cond:Cont}), we obtain $K(R_{v_n,z_n}) \rightharpoonup K(R_{v^*,z^*})$.
Since $\mathcal D(\cdot,g)$, $\mathcal E_1$ and $\mathcal E_2$ are all weakly lsc, we get $J(v^*,z^*) \leq \liminf_{n \to \infty} J(v_n, z_n) = \inf_{v,z} J(v,z)$.
Hence, $(v^*, z^*)$ is a minimizer for \eqref{eq:Indirect_Meta_elim}.
\end{proof}

Note that \eqref{eq:Indirect_Meta_elim} is non-convex.
Hence, we cannot expect uniqueness of the minimizer.
If $K$ is nonlinear, it appears sensible to require that it is completely continuous, i.e., that it maps weakly convergent sequences to strongly convergent ones.
In this case, we can relax the constraints for $\mathcal D$ by considering operators that are only lsc.
Next, we provide a result regarding the continuous dependence of minimizers for \eqref{eq:Indirect_Meta_elim} on the data $g \in Y$.

\begin{theorem}[Dependence on the data]\label{thm:stab}
Consider a sequence $g_n \to g$ in $Y$.
For each $n \in \mathbb{N}$, let  $(v_n, z_n)\in \mathcal V\times L^2(\Omega)$ be a minimizer of the functional $J_n \coloneqq J_{\lambda, g_n}$, where $\lambda = (\lambda_1, \lambda_2)$.
Then, there exists a subsequence of $\{v_n, z_n\}_{n \in \N}$ that converges weakly to a minimizer of $J \coloneqq J_{\lambda, g}$. 
\end{theorem}

\begin{proof}
As $(v_n, z_n)$ minimizes $J_n$, it holds
\begin{equation}
    \lambda_1 \mathcal E_1(v_n) + \lambda_2 \mathcal E_2(z_n) \leq J_n(v_n, z_n) \leq J_n(0,0) =\mathcal D\bigl(K(R_{0,0}), g_n\bigr) + C \to \mathcal D\bigl(K(R_{0,0}), g\bigr) + C,
\end{equation}
where the convergence follows by continuity of $\mathcal D$ in its second entry (Condition~\ref{cond:Cont}).
Similarly as in Theorem~\ref{thm:exist}, we choose a subsequence of $\{v_n\}_{n \in \N}$ with limit $v^*$ such that \smash{$f_n \coloneqq T \circ \varphi_{0,v_n}^{-1}(1,\cdot)$} converges strongly in $L^2(\Omega)$ to $f \coloneqq T \circ \varphi_{0,v^*}^{-1}(1,\cdot)$, i.e., $\{f_n\}_{n \in \N}$ is bounded.
By Condition~\ref{cond:InvTria}, it holds that
\begin{equation}\label{eq1.1}
\mathcal D\bigl(K(f_n+z_n),g\bigr) \leq \mathcal D\bigl(K(f_n+z_n),g_n\bigr) + C \mathcal D(g,g_n).
\end{equation} 
Now, Condition~\ref{cond:coercive} implies that also $\{z_n\}_{n \in \N}$ is bounded.
Hence, we may extract a weakly convergent subsequence $\{v_n, z_n\}_{n \in \N}$ with limit $(v^*, z^*) \in \mathcal V \times L^2(\Omega)$ and $R_{v_n,z_n} \weakly R_{v,z}$ in $L^2(\Omega)$.
By incorporation of \eqref{eq1.1}, we get that it holds 
\begin{equation}\label{eq1}
\mathcal D\bigl(K(R_{v^*,z^*}), g\bigr) \leq \liminf_{n \to \infty} \mathcal D\bigl(K(R_{v_n,z_n}), g\bigr) \leq \liminf_{n \to \infty} \mathcal D\bigl(K(R_{v_n,z_n}), g_n\bigr).
\end{equation} 
Now, let $(\tilde v, \tilde z)\in \mathcal V \times L^2(\Omega)$ be arbitrary.
By utilizing \eqref{eq1} and the weak lower semi-continuity of $\mathcal E_1$ and $\mathcal E_2$, we have
\begin{align}
J(v^*,z^*) \leq \liminf_{n \to \infty} J_n(v_n,z_n) \leq  \liminf_{n \to \infty} J_n(\tilde v, \tilde z).
\end{align}
The continuity of $\mathcal D$ implies $\mathcal D(K(R_{\tilde v, \tilde z}),g_n) \to \mathcal D(K(R_{\tilde v, \tilde z}),g)$.
Thus, it holds that 
\begin{equation}
    J(v^*,z^*) \leq \liminf_{n \to \infty} J_n(\tilde v, \tilde z)= J(\tilde v, \tilde z),
\end{equation}
which concludes the proof.
\end{proof}
We conclude our investigation with a convergence result for vanishing noise, provided that we use an appropriate parameter choice rule $\lambda = \gamma(\delta)$.
This enables us to approximate solutions of~\eqref{eq:operatorequation}.

\begin{theorem}[Convergence for vanishing noise]
\label{theorem: vanishing noise}
Let $T \in L^2(\Omega)$, $g\in Y$ and assume that there exists $(\hat{v},\hat{z}) \in \mathcal V\times  L^2(\Omega)$ with $K(R_{\hat{v},\hat{z}})=g$.
Further, assume that $\gamma_i\colon\mathbb{R}_{>0}\to \mathbb{R}_{>0}$, $i=1,2$, satisfy $\gamma_i(\delta)\to 0$, $\delta/\gamma_1(\delta) \to 0$ for $\delta \to 0$ and $\gamma_2(\delta)/\gamma_1(\delta) \to c>0$.
Choose a sequence $\{\delta_n\}_{n \in \N} \subset \mathbb{R}_{>0}$ that converges to zero, and a sequence $\{g_n\}_{n \in \N} \subset Y$ with $\mathcal D(g,g_n)\leq \delta_n$.
Furthermore, let $(v_n,z_n)$ be a minimizer of $J_n \coloneqq J_{\lambda_n,g_n}$ with $\lambda_n \coloneqq(\lambda_{1,n}, \lambda_{2,n})=(\gamma_1(\delta_n),\gamma_2(\delta_n))$ for each $n\in \mathbb{N}$.
Then there exists a subsequence of $\{v_n,z_n\}_{n \in \N}$ that converges weakly to a point $(v^*,z^*)$ with $K(R_{v^*,z^*})=g$.
\end{theorem}

\begin{proof}
For every $n \in \mathbb{N}$, it holds that
\begin{align}
\mathcal E_1(v_n) + \frac{\lambda_{2,n}}{\lambda_{1,n}} \mathcal E_2(z_n) \leq  \frac{1}{\lambda_{1,n}} J_n(v_n, z_n) \leq \frac{1}{\lambda_{1,n}} J_n(\hat{v},\hat{z}) 
&\leq \frac{1}{\lambda_{1,n}} \mathcal D(g,g_n) +  \mathcal E_1(\hat v) + \frac{\lambda_{2,n}}{\lambda_{1,n}} \mathcal E_2(\hat z) \notag\\
& \leq \frac{\delta_n}{\lambda_{1,n}} +  \mathcal E_1(\hat v) + \frac{\lambda_{2,n}}{\lambda_{1,n}} \mathcal E_2(\hat z).
\end{align}
Hence, $\{v_n\}_{n \in \N}$ and $\{\mathcal E_2(z_n)\}_{n \in \N}$ are bounded.
Further, $\mathcal D(K(R_{v_n.z_n}),g_n)$ is bounded, which as in Theorem~\ref{thm:stab} implies that $\{z_n\}_{n \in \N}$ is bounded.
Therefore, $\{v_n,z_n\}_{n \in \N}$ is bounded and possesses a weakly convergent subsequence (without relabeling) with limit $(v^*,z^*)$, for which the following estimate holds true
\begin{align}
\mathcal D\bigl(K(R_{v^*,z^*}),g\bigr)&\leq \liminf_{n \to \infty} \mathcal D\bigl(K(R_{v_n,z_n}),g\bigr)
			 \leq \liminf_{n \to \infty} \mathcal D\bigl(K(R_{v_n,z_n}),g_n\bigr) \leq \liminf_{n \to \infty} J_n(v_n,z_n)\notag\\
			 & \leq  \liminf_{n \to \infty} J_n(\hat{v},\hat{z})
			 \leq  \lim_{n \to \infty}  \delta_n + \lambda_{1,n}\mathcal E_1(\hat{v}) + \lambda_{2,n}\mathcal E_2(\hat{z})
			 =0.
\end{align}
Finally, by using Condition~\ref{cond:Zero}, we deduce $K(R_{v^*,z^*})=g$.
\end{proof}

\paragraph{Specific setting}
In our numerical experiments, we work with 2D images, i.e., $\Omega = (0,1)^2 \subset \R^2$, and we choose $\mathcal E_2(z) \coloneqq \TV(z)$.
Further, the data fidelity term $\mathcal D$ is chosen for all examples with synthetic data as $\mathcal D (f,g) = \tfrac12 \Vert f -g \Vert^2$, and the regularizer for $v$ as
\begin{equation}
\label{eq: rep diff reg}
\mathcal E_1 (v)\coloneqq \frac{1}{2}\int_0^T \int_\Omega \|B v(t,x)\|^2 \dx x \dx t,
\end{equation}
where $B$ is a differential operator such that Condition~\eqref{cond:coercive} is satisfied, i.e., $\mathcal E_1$ is coercive in $v$.
Since the Sobolev space $V = H_0^3(\Omega, \mathbb{R}^2)$ can be continuously embedded into $C^{1,0.5}(\Omega,\mathbb{R}^2)$, we have to choose a matrix $B$ that encodes all third-order derivatives in space. 
We want to remark that our numerical approach in Section \ref{sec:discrete} works for any regularizers of the form \eqref{eq: rep diff reg}, even if it is not coercive.

Now, the specification of problem \eqref{eq:Indirect_Meta_elim} reads
\begin{align}\label{eq:Indirect_Meta_speci}
\min_{(v,z ) \in \mathcal V \times L^2(\Omega)}  \frac12 \Vert K \circ R_{v,z} - g \Vert^2 + \frac{\lambda_1}{2}\int_0^T \int_\Omega \|B v(t,x)\|^2 \dx x \dx t + \lambda_2\mathcal \TV(z).
\end{align}
In case that $K\colon L^2(\Omega) \to Y$ is linear, bounded and does not vanish for constant functions, the assumptions underlying our theoretical investigations in the previous paragraph are satisfied:
\begin{itemize}
\item As $K$ is linear and bounded, it is also continuous and weak-weak-continuous.
\item Due to the norm properties, $\mathcal D$ is  weakly lsc and continuous in both entries.
Further, Condition~\ref{cond:Zero} and \ref{cond:InvTria} also follow from the norm properties.
\item We show Condition \ref{cond:coercive} explicitly.
Due to our choice of $B$, the regularizer $\mathcal E_1$ is coercive.
Hence, it remains to show the first part of the condition.
Let $\{f_n\}_n$ be bounded and let $\Vert z_n \Vert \to \infty$.
Then, either $\| P_0 z_n \| \to \infty$ or $\Vert Q_0 z_n \Vert \to \infty$.
In the first case, we have that $\mathcal E_2(z_n) \to \infty$ and the claim follows by positivity of $\mathcal D$.
For the second one, we get due to linearity of $K$ that
\begin{equation}
    K(f_n + z_n) = K(f_n + P_0 z_n) + K (Q_0 z_n),
\end{equation}
where the first term remains bounded.
Since $\Vert K (Q_0 z_n) \Vert \to \infty$ due to the assumption that $K$ does not vanish for constant functions, we conclude $\Vert K(f_n + z_n) - g \Vert \to \infty$ and the claim follows by positivity of $\mathcal E_2$.

\item The regularizers $\mathcal E_1$ and $\mathcal E_2$ are chosen such that they are weakly lsc.
\end{itemize}
\begin{remark}\label{rem:NCC1}
Although this is not covered theoretically, we use a normalized-cross-correlation-based distance $\mathcal D_{\mathrm{NCC}} \colon Y\setminus\{0\} \times Y\setminus\{0\} \to [0,1]$ given by
\begin{equation}
    \mathcal D_{\mathrm{NCC}}(f,g) = 1 - \frac{\langle f,g \rangle^2}{\Vert f \Vert_{Y}^2 \Vert g \Vert_{Y}^2}
\end{equation}
for our numerical experiments with real data as proposed in \cite{LNOS19}.
This modification is necessary as the gray value scale between template and target is often different for real data.
It holds that $\mathcal D_{\mathrm{NCC}}(f,g) =0$ only implies $f = c g$ with $c \in \R$, i.e., Condition~\ref{cond:Zero} is violated.  
Further, we do not have the required coercivity with respect to $z$ (Condition~\ref{cond:coercive}).
Hence, this setup is not covered by our theory.
Nevertheless, our numerical experiments indicate that $\mathcal D_{\mathrm{NCC}}$ combined with $\mathcal E_1 = \TV$ leads to good reconstructions.

If we choose \smash{$\mathcal{E}_2=\|\cdot\|_{L^2(\Omega)}^2$} instead, we can derive the same theoretical results as for $\mathcal D (f,g) = \Vert f -g \Vert^2$.
In this case, Condition \ref{cond:InvTria} and \ref{cond:coercive} are obsolete, as they are only required to infer boundedness of $z \in L^2(\Omega)$ from the boundedness of the objective, which now follows directly. The remaining conditions are met by $\mathcal D_{\mathrm{NCC}}$, except for Condition~\ref{cond:Zero}.
Hence, the convergence in Theorem~\ref{theorem: vanishing noise} holds only up to a scalar.
\end{remark}

\section{Numerical Approach}\label{sec:discrete}
In this section, we present a numerical scheme for solving \eqref{eq:Indirect_Meta_speci}. 
Our approach is based on the Lagrangian method developed in \cite{LNOS19,ManRut17} as well as the iPALM algorithm \cite{BST14,PS17}.
The actual implementation relies upon the FAIR toolbox \cite{Mod2004}.

So far, we have not specified the differential operator $B$ in \eqref{eq:Indirect_Meta_speci}.
As our approach is guided by the modular framework of \cite{LNOS19,ManRut17}, we could use \textit{curvature regularization}, defined by $B= \nabla_x$, or \textit{diffusion regularization}, where $B=\Delta_x$.
These choices correspond to the $H^1$ and $H^2$ semi-norm, respectively.
However, they do not satisfy the coercivity requirement, i.e., Condition~\ref{cond:coercive}.
Instead, we use the $H^3$ semi-norm, which has been proposed in \cite{LNOS19}.
Although our theoretical investigations for \eqref{eq:Indirect_Meta_speci} in Section~\ref{sec:proper} only hold for the 2D case, we provide the algorithm in general form, and also deploy it for 3D images later.
In the following, we briefly sketch the components of our approach.

\subsection{Discretization}\label{sec:discretization}
We pursue a \textit{discretize-then-optimize} strategy.
By partitioning every coordinate in $m$ blocks of length $h_X= 1/m$, the domain $(0,1)^d$ is split into $m^d$ equally sized cubes.
Then, the template $T \in L^2(\Omega)$ and the source $z \in L^2(\Omega)$ are both sampled at the cell-centered nodes \smash{$\mathbf{x_c} \in \mathbb{R}^{m^d}$}, resulting in discrete versions $\mathbf{T(x_c)}$ and \smash{$\mathbf{z(x_c)} \in \mathbb{R}^{m^d}$}, respectively.
Their values are interpolated by cubic B-splines if off grid values are required.
Further, the time domain $[0,1]$ is uniformly partitioned into $m_t$ units of length $h_t=1/m_t$.
Then, the velocity $v\colon [0,1] \times \Omega\to \mathbb{R}^d$ is sampled over cell-centered locations in space and at the nodes in time, resulting in a discrete velocity vector $\mathbf v \in \mathbb{R}^N$ with $N=d \cdot (m_t+1)\cdot m^d$.

\paragraph{Lagrangian solver for $\mathbf R_{\mathbf v, \mathbf z}$}
In order to compute the solution map $(\mathbf v, \mathbf z) \mapsto \mathbf R_{\textbf v, \mathbf z}$, we need to solve the flow equation \eqref{eq:Lag_mod}.
Here, every function $\varphi_{s,v}(\cdot,x_0)\colon[0, 1] \to \Omega$ can be interpreted as a trajectory of some particle with position $x_0$ at initial time $s$.
We compute $\mathbf R_{\mathbf v, \mathbf z}$ as follows:

\begin{enumerate}
    \item \textit{Computing the characteristics}: For numerically solving \eqref{eq:Lag_mod}, we employ a fourth order Runge--Kutta scheme (RK4).
    Since we require $\varphi_{1,\mathbf v}(0,\cdot)$, we solve \eqref{eq:Lag_mod} backwards in time with $N_t$ equidistant steps of size $\Delta t= -\frac{1}{N_t}$ and initial condition $\varphi_{1,\mathbf v}(1,\mathbf{x_c})= \mathbf{x_c}$.
    To simplify the notation, the remaining discussion is instead based on the explicit Euler scheme 
    \begin{equation}
    \label{eq:euler_step}
    \varphi_{1,\mathbf v}(t_{k+1}, \mathbf{x_c}) = \varphi_{1,\mathbf v}(t_k, \mathbf{x_c}) + \Delta t\, I \bigl( \mathbf{v} , t_k, \varphi_{1,\mathbf v}(t_{k}, \mathbf{x_c})\bigr),
    \end{equation}
    where $k=0,\ldots, N_t-1$ and $t_k=1 - k\Delta t$.
    Here, $I$ interpolates the velocity at time $t_k$ and transformed positions $\varphi_{1,\mathbf v}(t_{k}, \mathbf{x_c})$.
    This is necessary as the points $\varphi_{1,\mathbf v}(t_{k}, \mathbf{x_c})$ are in general not on the grid.
    Note that the time discretization parameters $N_t$ and $m_t$ differ in general:
    The first determines the accuracy of the ODE solver, whereas the second is related to the discretization.
    
    \item \textit{Deforming the template $T$:}
    Based on the output $\varphi_{1,\mathbf v}(0, \mathbf{x_c})$ of the RK4-scheme, we can evaluate the template $\mathbf T$ at the deformed grid using interpolation.
    
    \item \textit{Computing $\mathbf R_{\mathbf v, \mathbf z}$:}
    Finally, we add the source term $z$ and the deformed template $T$, i.e., $\textbf R_{\mathbf v,\mathbf z}(\mathbf{x_c}) = \mathbf T \circ \varphi_{1,\mathbf v}(0, \mathbf{x_c}) + \mathbf z(\mathbf{x_c})$.
\end{enumerate}
Note that actually all steps of this procedure are independent of the forward operator $K$, the data fidelity term $\mathcal D$ and the chosen regularizers $\mathcal E_1$, $\mathcal E_2$.
Further, the gradient $\nabla_{\mathbf{v}} \mathbf R_{\mathbf{v},\mathbf{z}}$ can be explicitly computed within the Runge--Kutta scheme, which is important for computational efficiency.

\paragraph{Forward operator}
Let us denote by $\mathbf{K}\colon \R^{m^d} \to \R^{M}$, $M \in \NN$, a finite-dimensional, Fr\'{e}chet differentiable approximation of the operator $K\colon L^2(\Omega, \R) \to Y$.
With the application to CT in mind, we discuss a discretization of the $d$-dimensional Radon transform.
More generally, if a discretization $\mathbf{K}$ of some operator $K$ is given, we can simply insert it into the model \eqref{eq:Indirect_Meta_speci}.

For given $\theta \in \mathbb{S}^{n-1}$ and  $s \in \mathbb{R}$, the Radon transform of $f\colon\mathbb{R}^d \to \mathbb{R}$ is defined pointwise by
\begin{equation}
    Rf(\theta,s) = \int_{\theta^\perp} f(s\theta+y)\dx y,
\end{equation}
where $\theta^\perp$ is the orthogonal complement of $\text{span}\{\theta\}$ \cite[Chap.~2]{N2001}.
The Radon transform is linear and thus also Fréchet differentiable.
For $f \in L^2(\Omega)$, the value $R(f) \in Y$ is a function that maps from the cylinder $\mathbb{S}^{d-1} \times \mathbb{R}$ to $\mathbb{R}$.
We discretize this cylinder as follows: Take $p \in \mathbb{N}$ directions in  $\mathbb{S}^{d-1}$.
For simplicity, we say that we take one measurement in each direction.
Furthermore, we take the interval $(0,1)$ instead of $\mathbb{R}$, and split it into $q\in \mathbb{N}$ equally sized cells of length $1/q$, as we have also done with $\Omega$.
Depending on the dimension $d$ and the diameter of $\Omega$, the intervals length requires adjustment.
Then, the data is sampled at cell-centered points $\mathbf{y_c}$ for each angle, resulting in vectors $\mathbf{g_i(\mathbf{y_c})}\in \mathbb{R}^q$, $i=1, \ldots,p$, and the entire data vector is represented as $\mathbf{g} \in \mathbb{R}^M$ with $M=p\cdot q$.
A discrete Radon transform is implemented for both CPUs and GPUs as part of the ASTRA toolbox \cite{AstraGPU,Astra2,Astra1}.

\paragraph{Data fidelity term and regularizers}
We discretize the data fidelity term $\mathcal D(x,y)=\frac{1}{2} \|x-y\|_2^2$ using the midpoint-rule for numerical integration, which results in
\begin{equation}
\label{discrete SSD}
\mathcal D_{SSD}(\mathbf{x}, \mathbf{y})= \frac{1}{2} h_Y (\mathbf{x}-\mathbf{y})^T( \mathbf{x}-\mathbf{y})
\end{equation}
with $h_Y=1/q$ and $\mathbf{x}, \mathbf{y} \in \mathbb{R}^M$, see also \cite[Chap.~6.2]{Mod2004}.
As we consider the Radon transform for few directions $\theta \in \mathbb{S}^{n-1}$ only, we disregard the necessary modifications related to the integration over the unit sphere.
Similarly, the discretization of the regularizer $\mathcal E_1$ is given by
\begin{equation}
\label{discrete Reg}
\mathcal E_1(\mathbf{v}) = \frac{1}{2} h_t h_X^d \mathbf{v}^T \mathbf{B}^T\mathbf{B}\mathbf{v},
\end{equation}
where $\mathbf{B}$ is a finite-difference counterpart of the chosen differential operator $B$.
To mitigate boundary effects caused by the discretization of $B$, we pad the spatial domain and impose zero Neumann boundary conditions.
For the $\TV$ regularizer, the norm is again discretized based on the midpoint-rule and the gradient is replaced by a finite difference counterpart, resulting in 
\begin{equation}\label{eq:disc_TV}
    \TV(\mathbf{z}) = h_X^d \sum_{i=1}^{ m^d} \Vert (\nabla_{h_X}(\mathbf{z}))_i\Vert.
\end{equation}
In our experiments, we employ backward differences and the boundary is extended by 0. 
To this end, we denote by $z\in \bigotimes_{i=1}^d\mathbb{R}^m$ the tensor representation of $\mathbf{z}\in \mathbb{R}^{m^d}$.
For any $i=1,...,m^d$, $\mathbf{z}_i$ corresponds to $z_{(i)}=z_{i_1,...,i_d}$ with $i_k=1,...m$.
Then, it holds that $(\nabla_{h_X}(\mathbf{z}))_i=  (\partial_k z_{(i)})_{k=1}^d$ with
\begin{equation}\label{eq:disc_grad}
    (\partial_{k} z_{(i)}) =\begin{cases} \frac{1}{h_x}\left( z_{i_1,..,i_{k}+1,..,i_d}-z_{(i)}\right) & \text{if } i_k < m,\\ 0 & \text{else.} \end{cases}
\end{equation}
\subsection{iPALM}
Putting all parts from Section \ref{sec:discretization} together, we obtain the discrete problem 
\begin{equation}\label{eq:Disc_Prob}
 \min_{(\mathbf{v},\mathbf{z})  \in \mathbb{R}^N \times \mathbb{R}^{m^d}} \mathcal D_{SSD} \bigl(\mathbf{K}(\mathbf R_{\mathbf v,\textbf z}),\mathbf{g}\bigr) + \frac{\lambda_1}{2} h_t h_X^d \mathbf{v}^T \mathbf{B}^T\mathbf{B}\mathbf{v} + \lambda_2h_X^d \sum_{i=1}^{d\cdot m^d} \Vert (\nabla_{h_X}(\mathbf{z}))_i\Vert.
\end{equation}
In the following, we deploy the \textit{inertial proximal alternating linearized minimization} (iPALM) algorithm \cite{BST14,PS17} for solving \eqref{eq:Disc_Prob}.
This scheme can solve generic problems of the form 
\begin{equation}
\label{PALM_problem}
\min_{x,y \in E_1\times E_2} \Psi(x,y) \coloneqq G_1(x) + G_2(y) +H(x,y),
\end{equation}
where $E_1$, $E_2$ are Euclidean spaces, $H \in C^1(E_1 \times E_2)$ and $G_i \in \Gamma_0(E_i)$, $i=1,2$, namely the proper, convex, and lsc functions on $E_i$.
The corresponding iterations are given by
\begin{align}
\begin{aligned}\label{PALM_x}
\bar{x}_k & = x_k+\alpha_k(x_k-x_{k-1})\\
x_{k+1} &= \text{prox}_{ \sigma_k^1 G_1} \Bigl( \bar{x}_k - \frac{1}{\sigma_k^1} \nabla_x H(\bar{x}_k, y_k) \Bigr)\\
\bar{y}_k & = y_k+\alpha_k(y_k-y_{k-1})\\
y_{k+1} &= \text{prox}_{ \sigma_k^2 G_2} \Bigl( \bar{y}_k - \frac{1}{\sigma_k^2} \nabla_y H(x_{k+1}, \bar{y}_k) \Bigr),
\end{aligned}
\end{align}
where $\prox_{\tau f} (x) = \argmin_{y} \frac12 \|x-y\|_2^2 + \tau f(y)$ is the proximal mapping of $f$.
The stepsizes $\sigma^1_k,\sigma_k^2>0$ are chosen according to the respective partial Lipschitz-constants of $\nabla H$, and $\alpha_k>0$ is an inertia parameter, which helps to escape from local minima and boosts the convergence speed.
In \cite{PS17}, the convergence of the iterations \eqref{PALM_x} is proven under the assumption that the objective has the \textit{Kurdyka–Łojasiewicz (KL) property}. 
Among others, this property holds for \textit{semi algebraic} functions, which include real polynomials, the $\| \cdot\|_p$-norm with rational, non-negative $p$, indicator-functions of semi-algebraic sets, as well as functions of the form $x \mapsto \sup\{g(x,y) \colon y \in S\}$, where $S$ is semi-algebraic and $g$ is a semi-algebraic function, see \cite{BST14}.
Furthermore, compositions of semi-algebraic functions are semi-algebraic.

\begin{theorem}[{\cite[Thm.~4.1]{PS17}}]\label{thm:palm}
	Let $E_1, E_2$ be Euclidean spaces, $H \in C^1(E_1 \times E_2)$, and
	$G_i\in \Gamma_0(E_i)$, $i=1,2$, such that $\Psi$ in \eqref{PALM_problem} has the KL property.
	Further, let all functions have finite infima.
	Assume that $\nabla H$ is locally Lipschitz continuous
	and that both $x_i \mapsto \nabla_{x_i} H(x_1,x_2)$ are globally Lipschitz, 
	where the constants $L_1(x_2), L_2(x_1)$ possibly depend on the fixed variable and are bounded on compact sets.
	Finally, assume 
	\begin{equation}
	    \sigma_k^1 =\frac{1 + 2 \alpha_k}{2(1-\alpha_k)}L_1(y_{k}) \quad \text{and} \quad \sigma_k^2 =\frac{1 + 2 \alpha_k}{2(1-\alpha_k)}L_2(x_{k+1})
	\end{equation}
	and $\alpha_k<0.5$ for every $k \in \N$.
	If the sequence generated by \eqref{PALM_x} is bounded, then it converges to a critical point.
\end{theorem}
\begin{remark}
    Although this is not supported by Theorem \ref{thm:palm}, it turned out that choosing $\alpha_k=(k-1)/(k-2)$, $\sigma_k^1 =L_1(y_{k})$ and $\sigma_k^2 =L_2(x_{k+1})$ works very well in practice \cite{PS17}.
    Unfortunately, the Lipschitz-moduli $L(x_k), L(y_k)$ are often unknown or difficult to compute.
    Instead, a backtracking scheme can be used to ensure this condition, see \cite{PS17} for details.
\end{remark}

Next, we want to apply iPALM to problem \eqref{eq:Disc_Prob}.
Both regularizers and the data fidelity term have the KL property since they are computed via composition of an affine function and a permissible norm function.
Further, the operation $(\mathbf{v}, \mathbf z) \mapsto \mathbf R_{\mathbf{v}, \mathbf z}$ is polynomial in $(\mathbf{v}, \mathbf z)$ for every component as the RK4 scheme and linear interpolators provide a composition of polynomial functions.
Hence, the objective in \eqref{eq:Disc_Prob} has the KL property.
We propose to use the splitting
\begin{align}
&H\colon\mathbb{R}^N \times \mathbb{R}^{ m^d}\to \mathbb{R},\quad (\mathbf{v}, \mathbf{z}) \mapsto \mathcal D_{SSD}\bigl(\mathbf{K}(\mathbf R_{\mathbf{v}, \mathbf z}),\mathbf{g}\bigr),\\
&G_1\colon \mathbb{R}^N \to \mathbb{R},\hspace{1.4cm} \mathbf{v} \mapsto  \frac{\lambda_1}{2} h_t h_X^d \mathbf{v}^T \mathbf{B}^T\mathbf{B}\mathbf{v},\\
&G_2\colon \mathbb{R}^{m^d} \to \mathbb{R},\hspace{1.3cm} \mathbf{z} \mapsto   \lambda_2 h_X^d \sum_{i}\Vert (\nabla_{h_X}(\mathbf{z}))_i\Vert.
\end{align}

The partial gradients of $H$ are given by
\begin{align}
\nabla_{\mathbf{v}} H(\mathbf{v},\mathbf{z}) &= h_Y \bigl( \mathbf{K}(\mathbf R_{\mathbf{v}, \mathbf z})-\mathbf{g}\bigr)^\tT {\partial } \mathbf{K}(\mathbf R_{\mathbf{v}, \mathbf z}) \frac{\partial}{\partial \mathbf{v}} \mathbf R_{\mathbf{v}, \mathbf z},\\
\nabla_{\mathbf{z}} H(\mathbf{v},\mathbf{z}) &= h_Y \bigl( \mathbf{K}(\mathbf R_{\mathbf{v}, \mathbf z})-\mathbf{g}\bigr)^\tT {\partial } \mathbf{K}(\mathbf R_{\mathbf{v}, \mathbf z}),
\end{align}
where $\partial \mathbf K$ refers to the Fréchet-derivative of the operator $\mathbf{K}$.
If $\mathbf{K}$ is linear, its derivative coincides with the operator itself.
The derivative $\frac{\partial}{\partial \mathbf{v}} \mathbf R_{\mathbf{v}, \mathbf z}$ of the solution map is given by
\begin{equation}
    \frac{\partial}{\partial \mathbf{v}} \mathbf R_{\mathbf{v}, \mathbf z} = \nabla_x \mathbf T\bigl(\varphi_{1,\mathbf v}(0, \mathbf{x_c})\bigr) \frac{\partial}{\partial \mathbf{v}} \varphi_{1,\mathbf v}(0, \mathbf{x_c}).
\end{equation}
Here, $\nabla_x \mathbf T$ is the gradient of the interpolated template $T$.
The derivative $\frac{\partial}{\partial \mathbf{v}} \varphi_{1,\mathbf v}(0, \mathbf{x_c})$ can be computed recursively within the ODE solver for \eqref{eq:Lag_mod}.
Exemplary, we obtain for the explicit Euler scheme \eqref{eq:euler_step} that 
\begin{align}
\frac{\partial}{\partial \mathbf{v}} \varphi_{1,\mathbf v}(t_{k+1}, \mathbf{x_c}) =\frac{\partial}{\partial \mathbf{v}} \varphi_{1,\mathbf v}(t_{k}, \mathbf{x_c})
&+\Delta t \frac{\partial}{\partial \mathbf{v}} I \bigl( \mathbf{v} , t_k, \varphi_{1,\mathbf v}(t_{k}, \mathbf{x_c})\bigr)\notag\\
&+\Delta t \frac{\partial}{\partial \varphi} I \bigl( \mathbf{v} , t_k, \varphi_{1,\mathbf v}(t_{k}, \mathbf{x_c})\bigr) \frac{\partial}{\partial \mathbf{v}} \varphi_{1,\mathbf v}(t_{k}, \mathbf{x_c})
\end{align}
for all $k=0,\ldots,N_t-1$, see also \cite{Mod2009}.

Further, we require the proximal mappings of $G_1$ and $G_2$.
For $G_1$, it holds that
\begin{equation}
    \prox_{\sigma_1 G_1}(\mathbf v) = \argmin_{\mathbf{x} \in \mathbb{R}^N} \frac{1}{2\sigma_1}\|\mathbf{x}-\mathbf v\|^2 + \frac{1}{2} \gamma h_t h_X^d \mathbf{x}^T \mathbf{B}^T\mathbf{B}\mathbf{x},
\end{equation}
where the minimum is determined by the linear system of equations
$(\text{Id}+\sigma_1 \gamma h_t h_X^d \mathbf{B}^T\mathbf{B})\mathbf{x}= \mathbf{v}$.
This system is sparse and efficiently solvable with a preconditioned conjugate gradient method.
Regarding $G_2$, we need to compute $\prox_{\sigma_2 \TV}$ with the discrete TV \eqref{eq:disc_TV}.
As outlined in \cite{BL18}, this can be done using the primal dual hybrid gradient method \cite{CP11}.

The step-sizes in the iPALM scheme \eqref{PALM_x} are chosen as the partial Lipschitz constants of $\nabla_\mathbf{v} H(\mathbf{v}, \mathbf{z})$ and $\nabla_\mathbf{z} H(\mathbf{v}, \mathbf{z})$, respectively.
For a linear $\mathbf{K}$, it holds $\nabla_\mathbf{z} H(\mathbf{v}, \mathbf{z})=  h_Y \mathbf{K}^T \mathbf{K}(\mathbf{f}(\mathbf R_{\mathbf{v}, \mathbf z})-\mathbf{g})$. 
Hence, $\mathbf{z} \to \nabla_\mathbf{z} H(\mathbf{v}, \mathbf{z})$ is Lipschitz with $ L_2(\mathbf{v})=h_Y \|\mathbf{K}^T \mathbf{K}\|$ for every $\mathbf{v}$.
Unfortunately, an upper bound for $L_1$ cannot be derived explicitly.
Hence, we have to rely on backtracking instead.
\begin{remark}[Normalized-cross-correlation-based distance]
\label{rem:NCCnum}
For $\mathcal D_{\mathrm{NCC}}$, we use the discretization
\begin{equation}
    \mathcal D_{\mathrm{NCC}}(\mathbf{x}, \mathbf{y}) = 1 - \frac{(\mathbf{x}^{\top}\mathbf{y})^{2}}{\Vert \mathbf{x} \Vert^{2} \Vert \mathbf{y} \Vert^{2}},
\end{equation}
for which the derivative is given by
\begin{equation}\label{eq:deriv_NCC}
	\frac{\partial}{\partial \mathbf{x}} \mathcal D_{\mathrm{NCC}}(\mathbf{x}, \mathbf{y}) = - 2\frac{(\mathbf{x}^{\top} \mathbf{y}) \mathbf{y}}{\lVert \mathbf{x} \rVert^{2} \lVert \mathbf{y} \rVert^{2}} + 2\frac{(\mathbf{x}^{\top} \mathbf{y})^{2} \mathbf{x}}{\lVert \mathbf{x} \rVert^{4} \lVert \mathbf{y} \rVert^{2}}.
\end{equation}
Incorporating \eqref{eq:deriv_NCC}, we can perform the according gradient steps in the iPALM scheme \eqref{PALM_x}.
Then, also a line search for estimating the Lipschitz constant with respect to $z$ is needed.
As discussed in Remark \ref{rem:NCC1}, $\mathcal D_{\mathrm{NCC}}$ should be paired with an $L^2$-regularizer for the source variable to get theoretical guarantees.
Then, we obtain $\mathcal{E}_2(z)=\|z\|_{L^2(\Omega)}^2$, which we discretize similarly as $\mathcal D_{\mathrm{SSD}}$, namely as
\begin{equation}
    G_2\colon \mathbb{R}^{m^d}\to \mathbb{R}, \quad \mathbf{z}\to \frac{1}{2} h_X\mathbf{z}^T\mathbf{z}.
\end{equation}
This results in the proximal operator
\begin{equation}
    \prox_{\sigma G_2}(\mathbf{z})= \argmin_{\mathbf{x} \in \mathbb{R}^{m^d}} \frac{1}{2\sigma} \|\mathbf{x}-\mathbf{z}\|^2 +\frac{1}{2} h_X \mathbf{x}^T \mathbf{x} =\frac{1}{1+\sigma/h_X}\mathbf{z}.
\end{equation}
\end{remark}

\paragraph{Multi-level approach and post-processing}
Due to the non-convexity of \eqref{eq:Disc_Prob}, we have to cope with local minima.
To avoid being trapped in one, we follow a multi-level strategy \cite{Mod2009}, which also helps to reduce the computational cost.
Its different levels refer to different resolutions of the template and the target.
We apply iPALM at each level, starting with the coarsest resolution.
Each computed minimizer is bilinearly interpolated to the next finer scale to serve as initialization.

This approach requires multi-level versions of the operator $\mathbf K$ and a method for downsampling the measurements $g$.
If these are not available, iPALM can still be performed with only one scale.
For the 2D Radon transform, multi-level versions of the operator $\mathbf K$ can be obtained with any backend that takes the discretization of the measurement geometry as an input.
More precisely, assume that the number of grid cells used to discretize $\Omega \subset \R^2$ at the finest level is $m = 2^l$, $l \in \N$.
In our experiments, we set the number of cells for discretizing the measurement domain $(0, 1)$ at level $k \leq l$ to $q^{(k)} = 1.5 \cdot 2^k$ and the length of each cell to \smash{$h^{(k)}_Y = 1/{q^{(k)}}$}.
Then, a multilevel representation of each measurement $\mathbf{g}_{i}$, $i \le p$, at cell-centered grid points \smash{$\mathbf{y}_{j} = (j - 1/2) h_{Y}^{(k - 1)}$} is given by
\begin{equation}
	\mathbf{g}_{i}^{(k - 1)}(\mathbf{y}_{j}) = \left( \mathbf{g}_{i}^{(k)}(\mathbf{y}_{j}) + \mathbf{g}_{i}^{(k)}\bigl(\mathbf{y}_{j} + h_{Y}^{(k)}\bigr) \right) / 4,
\end{equation}
where the denominator arises from averaging over two neighboring grid points and dividing the edge length of the image domain $\Omega$ in each coordinate direction in half.

Additionally, after applying iPALM, we deploy the second-order inexact Gauss--Newton method from \cite{LNOS19,ManRut17} for refining the velocity $\mathbf v$.
This method utilizes the same discretization and Lagrangian solver that is used for iPALM.
Due to the linearity of $\mathbf{K}$ and the structure of $\mathcal D_{\mathrm{SSD}}$, we can indeed fix $\mathbf z$ and consider the data $\tilde{\mathbf{g}}= \mathbf{g}-\mathbf{K}(\mathbf{z})$ within the corresponding LDDMM Gauss--Newton scheme for $\mathbf v$.
We observed that this can compensate the slower convergence rate of iPALM.
Finally, note the described modifications can also be applied within the setting of Remark \ref{rem:NCCnum}.
\begin{remark}
    There is a vast literature on alternating minimization and forward-backward schemes with variable metrics \cite{BFO19,BLPP17,CPR16,FGP15,STP17}, which attempt to improve the convergence speed.
    For the iPALM scheme \eqref{PALM_x}, our simulations indicate that an adaption of the metric is most promising for $\mathbf v$.
    To pursue this idea further, we choose a different splitting of \eqref{eq:Disc_Prob} and add $G_1$ to $H$ instead.
    Hence, $G_1=0$ and $\prox_{G_1} = \Id$ is independent of the chosen metric for the coordinate $\mathbf v$.
    There are different strategies for constructing metrics such that convergence guarantees can be obtained, e.g., minimize-maximize strategies \cite{CPR16} or sparse approximations of the Hessian \cite{BFO19}.
    For a small benchmark, we deployed the same metric for $\mathbf v$ as proposed in \cite{LNOS19}, and solved \eqref{eq:Disc_Prob} with a variable metric version of PALM, i.e., without the inertia steps.
    The sufficient decrease of the objective with respect to $\mathbf v$ is ensured with the same Armijo line search as in \cite{LNOS19}.
    Experimentally, this has led to similar results as our post-processing scheme.
    A more thorough comparison of the approaches could be an interesting direction of future research.
\end{remark}

\section{Numerical Examples} \label{sec:experiments}
\begin{figure}[t]
\centering
\begin{subfigure}{.32\textwidth}
  \centering
  \includegraphics[width=\linewidth]{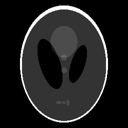}
  \caption{Template image $\mathbf T$.}  
\end{subfigure}
\begin{subfigure}{.32\textwidth}
  \centering
  \includegraphics[width=\linewidth]{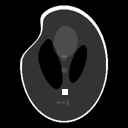}
  \caption{Target image $\mathbf U$.}  
\end{subfigure}
\begin{subfigure}{.32\textwidth}
  \centering
  \includegraphics[width=\linewidth]{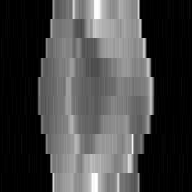}
  \caption{Noisy data $\mathbf g$.}  
\end{subfigure}
\vspace{.1cm}

\begin{subfigure}{.32\textwidth}
  \centering
  \includegraphics[width=\linewidth]{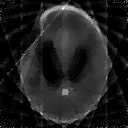}
  \caption{$L^2$-TV model \eqref{eq:L2_TV}.
  } 
  \label{phantom d} 
\end{subfigure}
\begin{subfigure}{.32\textwidth}
  \centering
  \includegraphics[width=\linewidth]{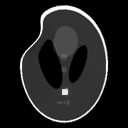}
  \caption{Our method with $E_2 = \TV$.}
  \label{phantom a}  
\end{subfigure}
\begin{subfigure}{.32\textwidth}
  \centering
  \includegraphics[width=\linewidth]{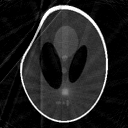}
  \caption{Our method with $E_2 = \Vert \cdot \Vert^2$.}  
  \label{phantom b}
\end{subfigure}
\caption{Reconstruction for a cartoon like image with 10 measurement angles using our model \eqref{eq:Indirect_Meta_speci} and the $L^2$-TV model \eqref{eq:L2_TV} as a baseline comparison.
The reconstruction $\mathbf R$ obtained with \eqref{eq:Indirect_Meta_speci} can be decomposed into a deformation and a source part.}
\label{cartoon}
\end{figure}
Here, we present numerical results for our model \eqref{eq:Indirect_Meta_speci}, discretized and solved numerically as described in Section~\ref{sec:discrete}.
The code for all experiments is available on Github\footnote{https://github.com/anttop/FAIR.m}.
Since we are mainly interested in a proof-of-concept, we only investigate CT and leave other inverse problems for future work.
Our examples mostly rely on synthetic data generated from target images $\mathbf U$ of size $128 \times 128$ with range $[0,1]$.
A corresponding sinogram $\mathbf g$ is obtained by applying the Radon transform with ten equally distributed angles in $[0,180]$ to $\mathbf U$.
Then, we add $5 \%$ Gaussian noise to get $\mathbf{g}$.
Further, we also include one example with real data from a CT scanner.
For all examples, we use a time-dependent velocity field $\mathbf v$ with a single time point, and 5 steps in the Runge--Kutta method that solves the associated equation \eqref{eq:Lag_mod}.
The multi-level procedure starts with resolution $32 \times 32$  at the coarsest scale.
Throughout this section, all parameters are optimized via grid search.

For the first experiment, the template $\mathbf T$ is chosen as the Shepp-Logan phantom and the target $\mathbf U$ is a diffeomorphically deformed version of $\mathbf T$ with an additional small white square, see Figure~\ref{cartoon}.
As TV regularization favors constant areas, this image pair is well-suited for our model \eqref{eq:Indirect_Meta_speci}, and an almost perfect reconstruction is expected.
Due to the additional structure, a good reconstruction with purely diffeomorphic approaches such as \cite{CO18,LNOS19} is impossible.
The regularization parameter $\lambda_1$ splits into $\lambda_1=[\lambda_1^1,\lambda_1^2, \lambda_1^3]$ for the spatial, temporal and $L^2(\Omega, \mathbb{R})$ regularization of $\mathbf v$, respectively.
We have chosen $\lambda_1=[0.001,0.001,10^{-6}]$ and $\lambda_2= 0.1$.
Recall that the reconstruction $\mathbf R$ can be decomposed into a deformation part $\mathbf T(\varphi^{-1}(1,\mathbf{x_c}))$ and a source part $\mathbf z$.
We observe that the main structure is reconstructed as deformation of the template $\mathbf T$, whereas $\mathbf z$ reconstructs the additional square, see Figure~\ref{cartoon}.
This is indeed the expected behavior for properly chosen parameters.
As comparison, we included a reconstruction with the standard $L^2$-TV model \cite{ROF92}, i.e., the solution of
\begin{equation}\label{eq:L2_TV}
    \min_{\mathbf{R}}\|\mathbf{K}(\mathbf{R})-\mathbf{g}\|_2^2 + \lambda \TV(\mathbf{R}),
\end{equation}
where we have chosen $\lambda= 0.1$ based on a grid search.
For \eqref{eq:L2_TV}, we observe the typical reconstruction artifacts related to the Radon transform, i.e., rays crossing the reconstruction.
If we increase $\lambda$ to avoid these artifacts, the reconstruction looses details.
Last, we also included a reconstruction with $E_2 = \Vert \cdot \Vert^2$ and $\lambda_2= 0.0001$ instead of $E_2 = \TV$.
This regularization has been investigated for the metamorphosis model \eqref{eq:Indirect_Meta} in \cite{GO18}.
As expected, we get similar artifacts as in \eqref{eq:L2_TV} with small $\lambda$.

In our second experiment, we deal with a pair of images that contain finer details.
More precisely, we have chosen an artificial brain image \cite{GLPU12} as template $\mathbf T$, which is diffeomorphically deformed into the target $\mathbf U$ \cite{LNOS19}.
Further, we added a structure in $\mathbf U$ to get a non-diffeomorphic setting.
We also varied the smoothness of this structure, which leads to two sub-experiments.
In the first one, we added a circle with constant intensity and in the second one a 2D Gaussian.
The obtained results for \eqref{eq:Indirect_Meta_speci} with $\lambda_1=[0.001,0.1,10^{-6}]$ and $\lambda_2=0.2$ are depicted in Figure~\ref{brain_circle}.
Note that we increased the contrast for the error maps in order to improve the visibility.
Our method is able to reconstruct all the major structures in $\mathbf U$.
Most of the errors occur at the boundaries of the structures, which is partially due to the employed interpolation.
Similarly as reported in \cite{LNOS19}, our approach struggles with the swirl in the middle of $\mathbf U$.
This is most likely due to the large, almost non-diffeomorphic deformation and the limited amount of data.
If we compare the $\mathbf z$ for the two sub-experiments, we observe that the reconstruction of the Gaussian is worse, which is not surprising as $\TV$-regularization favors piece-wise constant images.
In absence of topological changes, \cite{LNOS19} demonstrated that a LDDMM-based approach  without the source term can yield satisfactory results.
A reconstruction with our method for $\mathbf{U}$ without the additional structure is provided in  Figure~\ref{fig: no_top}. 
For $\lambda_1=[0.001,0.001,10^{-5}]$  and $\lambda_2=1$, the reconstruction consists only of a  deformation part. 
However, if $\lambda_2$  is chosen too small, artifacts related to the source $z$ appear.
\begin{figure}
\centering
\begin{subfigure}{.32\textwidth}
  \centering
  \includegraphics[width=\linewidth]{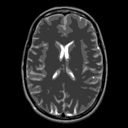}
  \caption{Template image $\mathbf T$.} 
  \label{template} 
\end{subfigure}
\begin{subfigure}{.32\textwidth}
  \centering
  \includegraphics[width=\linewidth]{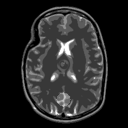}
  \caption{Target image $\mathbf U_1$.} 
  \label{target} 
\end{subfigure}
\begin{subfigure}{.32\textwidth}
  \centering
  \includegraphics[width=\linewidth]{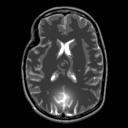}
  \caption{Target image $\mathbf U_2$.} 
  \label{target_g} 
\end{subfigure}
\vspace{.1cm}

\begin{subfigure}{.32\textwidth}
  \centering
  \includegraphics[width=\linewidth]{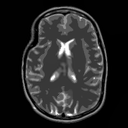}
  \caption{Reconstruction of $\mathbf{U}_1$.} 
  \label{result} 
\end{subfigure}
\begin{subfigure}{.32\textwidth}
  \centering
  \includegraphics[width=\linewidth]{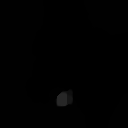}
  \caption{Source $\mathbf z_1$.}  
  \label{deformation}
\end{subfigure}
\begin{subfigure}{.32\textwidth}
  \centering
  \includegraphics[decodearray={0 2.5},width=\linewidth]{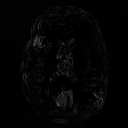}
  \caption{Error: max $0.3$, $\text{SSD} =0.02$.} 
  \label{error}
\end{subfigure}
\vspace{.1cm}

\begin{subfigure}{.32\textwidth}
  \centering
  \includegraphics[width=\linewidth]{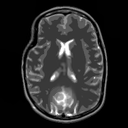}
  \caption{Reconstruction of $\mathbf{U}_2$.} 
  \label{result_g} 
\end{subfigure}
\begin{subfigure}{.32\textwidth}
  \centering
  \includegraphics[width=\linewidth]{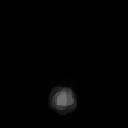}
  \caption{Source $\mathbf z_2$.}  
  \label{deformation_g}
\end{subfigure}
\begin{subfigure}{.32\textwidth}
  \centering
  \includegraphics[decodearray={0 2.5},width=\linewidth]{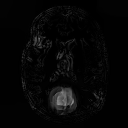}
  \caption{Error: max $0.52$, $\text{SSD} =0.03$.} 
  \label{error_g} 
\end{subfigure}
\caption{Artifical brain image with data from $10$ equally distributed angles in $[0,180]$.
The first target contains an additional circle, which is well-suited for $\TV$ regularization of $z$.
The second one contains an additional 2D Gaussian, which is more challenging for this setting.}
\label{brain_circle}
\end{figure}
\begin{figure}
\begin{subfigure}{.32\textwidth}
  \centering
  \includegraphics[width=\linewidth]{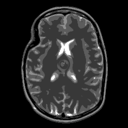}
  \caption{Target image $\mathbf U$.} 
  \label{target nt} 
\end{subfigure}
\begin{subfigure}{.32\textwidth}
  \centering
  \includegraphics[width=\linewidth]{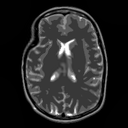}
  \caption{Result for $\lambda_2=1$.} 
  \label{result nt}
\end{subfigure}
\begin{subfigure}{.32\textwidth}
  \centering
  \includegraphics[width=\linewidth]{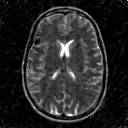}
  \caption{Result for $\lambda_2=0.0001$.}  
  \label{res small nt}
\end{subfigure}
\caption{Diffeomorphic counterpart to Figure \ref{brain_circle}, where $\mathbf{U}$ does not contain an additional structure.}
\label{fig: no_top}
\end{figure}

Next, we comment
on the robustness with respect to parameter changes.
As the comparative examples in Figure~\ref{fig: no_top} and~\ref{fig:lotus} show, the 
choice of parameters significantly impacts the quality of the reconstruction.
For the images from Figure~\ref{brain_circle}, the influence of the parameter choice in terms of SSD error and SSIM value is provided in Tabular~\ref{tab: SSIM}.
Some corresponding reconstructions are given in Figure~\ref{fig:ExamplesPara}.
Changing each parameter by an order of magnitude leads to clearly visible changes in the reconstruction, see Figure~\ref{fig:ExamplesPara}.
In particular, too large regularization parameters can repress the effect of their respective component.
In contrast, changes on a smaller scale lead to robust reconstruction results.
For the other examples from this section, we observed a similar behavior, although the precise scale depends on the underlying set of images.
\begin{figure}
    \centering
    \begin{subfigure}{0.24\textwidth}
       \includegraphics[width=\textwidth]{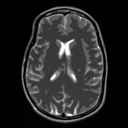} 
        \caption{$\lambda_1=10$, $\lambda_2=2$}
        \label{fig: wrong a}
    \end{subfigure}
    \begin{subfigure}{0.24\textwidth}
       \includegraphics[width=\textwidth]{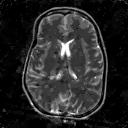} 
         \caption{$\lambda_1=0.01$, $\lambda_2=0.001$}
        \label{fig: wrong b} 
    \end{subfigure} 
    \begin{subfigure}{0.24\textwidth}
       \includegraphics[width=\textwidth]{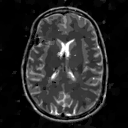} 
        \caption{$\lambda_1=0.1$, $\lambda_2=0.05$}
        \label{fig: right a}
    \end{subfigure}
    \begin{subfigure}{0.24\textwidth}
       \includegraphics[width=\textwidth]{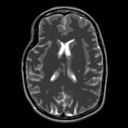} 
        \caption{$\lambda_1=0.1$, $\lambda_2=0.5$}
        \label{fig: right b}
    \end{subfigure}
    \caption{Reconstructions of an artificial brain image with various parameters.
    Drastic changes lead to artifacts (\ref{fig: wrong a}-\ref{fig: wrong b}), whereas
    smaller ones lead to similar reconstructions (\ref{fig: right a}-\ref{fig: right b}) as in Figure~\ref{result}.}\label{fig:ExamplesPara}
\end{figure}
\begin{table}
    \centering
    \begin{tabular}{r|p{2cm}| p{2cm}|p{2cm}|p{2cm}}
    \diagbox{$\lambda_1$}{$\lambda_2$} 
                  &$0.02$    & $0.2$    &$2$    & $20$  \\
       \hline
       $0.01$     &SSD: $0.0073$ \newline SSIM: $0.28$   
                  &SSD:  $0.0066$ \newline SSIM: $0.30$
                  &SSD: $0.0061$  \newline SSIM: $0.31$ 
                  &SSD:$ 0.0054$  \newline SSIM: $0.34$ 
                  \\
                  \hline
        $0.1$     &SSD: $0.0019$ \newline SSIM: $0.44$ 
                  &SSD:  $0.0007$\newline SSIM: $0.54$ 
                  &SSD: $0.0011$ \newline SSIM: $0.51$ 
                  &SSD: $ 0.0013$\newline SSIM: $0.50$ 
                  \\
                   \hline
        $1$      &SSD: $0.0027$\newline SSIM: $0.42$ 
                 &SSD:  $0.0008$ \newline SSIM: $0.54$ 
                 &SSD: $0.0012$ \newline SSIM: $0.51$ 
                 &SSD: $ 0.0012$\newline SSIM: $0.51$ 
                 \\
                  \hline
        $10$     &SSD: $0.0042$ \newline SSIM: $0.37$ 
                &SSD:  $0.0017$ \newline SSIM: $0.49$ 
                &SSD: $0.0016$ \newline SSIM: $0.51$ 
                &SSD: $ 0.0020$\newline SSIM: $0.48$ 
                \\
                 \hline
        $100$    &SSD: $0.0057$\newline SSIM: $0.32$ 
                &SSD:  $0.0028$\newline SSIM: $0.44$ 
                &SSD: $0.0033$ \newline SSIM: $0.45$ 
                &SSD:  $ 0.0033$\newline SSIM: $0.45$ 
    \end{tabular}
    \caption{SSD-error and SSIM for various parameters with the images from Figure~\ref{brain_circle}.
    For greater clarity, the deformation regularization parameter is chosen as $\lambda=\lambda_1 [0.001,0.1, 10^{-6}]$.}
    \label{tab: SSIM}
\end{table}

For our third experiment, see Figure~\ref{fig:hands}, we have chosen two X-ray images that are not diffeomorphic to each other and contain some noise structures.
Finding the correct deformation for this pair is challenging as the deformation is relatively large and irregular.
In this experiment, we compare the third-order regularizer ($\lambda_1=[0.5,0.1,10^{-6}]$, $\lambda_2=1$) with the curvature one from \cite{ManRut17} ($\lambda_1=[6,6]$, $\lambda_2=1$), which imposes less regularity on $\mathbf v$.
For both choices, we observe that the model struggles to bend the hand to correctly match the target.
This is most noticeable at the right corner of the palm and at some of the fingers.
For the third-order regularizer, we also notice that the fingers are not spread sufficiently from each other.
The additional or disappearing structures are very fine, such as the noise beside the hands or the slight change in intensity on the edges of the bones.
Hence, these are not well-suited for TV regularization and not reconstructed by our method.
However, even if we would use a different $E_2$, it is questionable if the amount of data suffices to reconstruct these structures.
More precisely, reconstructing fine details without sufficient data or an appropriate prior is in general impossible.
Nevertheless, this experiment shows that our model \eqref{eq:Indirect_Meta_speci} allows to align image pairs with large deformations between them, even under the presence of noise.

In our last synthetic experiment, we demonstrate that our method is in principal also applicable to 3D CT reconstruction problems, see Figure~\ref{fig:3d}.
To this end, we use an image pair from~\cite{Mod2009} and the same 10 simulated 2D measurements as in \cite{LNOS19}, which correspond to an rotation around the third coordinate axis with angles equally distributed in $[0,180]$.
Due to the problem size, we use curvature regularization on $\mathbf v$ with $\lambda_1=[0.07,0.07]$ for the spatial and temporal components, respectively, and for the source $\mathbf z$ we use the regularization parameter $\lambda_2=0.01$.
Overall, we obtain a satisfying reconstruction with a SSIM of $0.9060$, which is significantly better than the SSIM  of $0.8807$ obtained by \cite{LNOS19}.

In our final experiment, we tackle real CT data from a lotus root cross-section \cite{BubHauHuoRimSil16}.
Since the recorded data is dense, the underlying target $\mathbf U$ can be reconstructed via filtered back-projection.
Retroactively, we deformed the computed target $\mathbf U$ and removed a hole in the lotus root to obtain a template $\mathbf T$ with a different topology, see Figure~\ref{fig:lotus}.
To get a sparse setting, we subsampled $\mathbf g$ with 12 uniformly distributed angles in $[0,180]$.
As we are already dealing with real data, no additional noise is added. 
Unfortunately, the intensity range of the given sinogram does not match the range of our Radon transform operator $\mathbf{K}$.
Therefore, $\mathcal D_{\mathrm{SSD}}$ is no sensible choice for comparing the real data with the simulated data $\mathbf{K}(\mathbf R)$.
Instead, we use $\mathcal D_{\mathrm{NCC}}$, which  is invariant to the scaling of the images' intensity, see Remark~\ref{rem:NCC1}.
Although our theoretical results in Section~\ref{sec:proper} do not apply in this setting, we obtain satisfying numerical results using $\lambda_1=[0.01,0.01,10^{-6}]$ and $\lambda_2=0.001$, see Figure~\ref{fig:lotus}.
More precisely, our method manages to find the main deformation and the additional hole.
In this real data setting, we also investigate how a reconstruction with $E_2 = \Vert \cdot \Vert^2$ instead of $E_2 = \TV$ performs.
Here, our theoretical results from Section~\ref{sec:proper} hold again.
To this end, we included two reconstructions corresponding to $\lambda_2=0.1$ and $\lambda_2=0.01$, respectively.
In the first reconstruction, we have almost no artifcats, but can also only guess the new hole.
For the second one, the hole can be clearly seen, but the artifcats are much stronger.
Either way, the results are inferior to those generated with $E_2 = \TV$.
\begin{figure}
\centering
\begin{subfigure}{.31\textwidth}
  \centering
  \includegraphics[width=\linewidth]{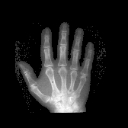}
  \caption{Template image $\mathbf T$.} 
  \label{template Hands} 
\end{subfigure}
\begin{subfigure}{.31\textwidth}
  \centering
  \includegraphics[width=\linewidth]{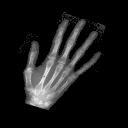}
  \caption{Target image $\mathbf U$.} 
  \label{target Hands} 
\end{subfigure}
\begin{subfigure}{.31\textwidth}
  \centering
  \includegraphics[width=\linewidth]{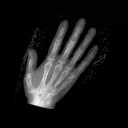}
  \caption{Reconstruction (third order).} 
  \label{result Hands} 
\end{subfigure}
\vspace{.1cm}

\begin{subfigure}{.31\textwidth}
  \centering
  \includegraphics[width=\linewidth]{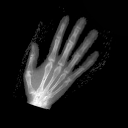}
  \caption{Reconstruction (curvature).\\ \quad}
  \label{deformation Hands}
\end{subfigure}
  \begin{subfigure}{.31\textwidth}
  \centering
  \includegraphics[width=\linewidth]{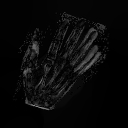}
  \caption{Error (third order):\\\strut{}
  \hspace{.3cm} max $0.54$, $\text{SSD} =0.05$.} 
  \label{fig:error_hands} 
\end{subfigure}
\begin{subfigure}{.31\textwidth}
  \centering
  \includegraphics[width=\linewidth]{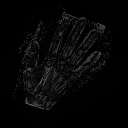}
  \caption{Error (curvature):\\\strut{} \hspace{.3cm} max $0.44$, $\text{SSD} =0.05$.}
  \label{source Hands} 
\end{subfigure} 
\caption{Reconstruction of a human hand \cite{Mod2009} with measurements for 10 angles in $[0,180]$.
The underlying deformation is relatively large and the images are non-diffeomorphic.
Especially the small noise structures outside of the bone areas are hard to reconstruct.}
\label{fig:hands}
\end{figure}
\begin{figure}
    \centering
    \begin{subfigure}{0.3\textwidth}
        \includegraphics[width=\textwidth]{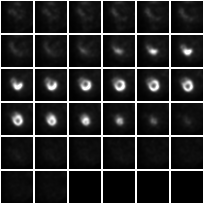}
        \caption{Template image $\mathbf T$.}
        \label{3d template}
    \end{subfigure}
    \begin{subfigure}{0.3\textwidth}
        \includegraphics[width=\textwidth]{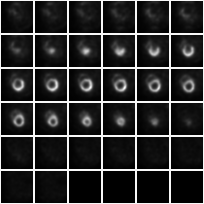}
        \caption{Target $\mathbf U$.}
        \label{3d target}
    \end{subfigure}
     \begin{subfigure}{0.3\textwidth}
        \includegraphics[width=\textwidth]{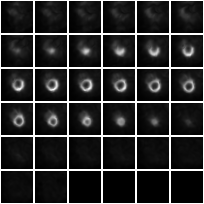}
        \caption{Reconstruction of $\mathbf U$.}
        \label{3d result}
    \end{subfigure}
    \caption{Reconstruction of a 3D volume using only ten measurement directions.
    Each image depicts a slice of the volume along the third coordinate axis.}
    \label{fig:3d}
\end{figure}

\begin{figure}
\centering
\begin{subfigure}{.32\textwidth}
  \centering
  \includegraphics[width=\linewidth]{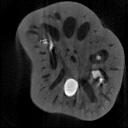}
  \caption{Template image $\mathbf T$.} 
  \label{template lotus} 
\end{subfigure}
\begin{subfigure}{.32\textwidth}
  \centering
  \includegraphics[width=\linewidth]{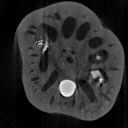}
  \caption{Target image $\mathbf U$.} 
  \label{target lotus} 
\end{subfigure}
\begin{subfigure}{.32\textwidth}
  \centering
  \includegraphics[width=\linewidth]{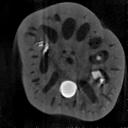}
  \caption{Reconstruction.} 
  \label{result lotus} 
\end{subfigure}
\vspace{.1cm}

\begin{subfigure}{.32\textwidth}
  \centering
  \includegraphics[width=\linewidth]{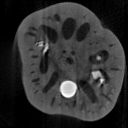}
  \caption{Deformation part only.}  
  \label{deformation lotus}
\end{subfigure}
\begin{subfigure}{.32\textwidth}
  \centering
  \includegraphics[width=\linewidth]{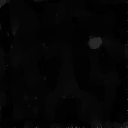}
  \caption{Source $\mathbf z$.}
  \label{source lotus} 
\end{subfigure} 
\begin{subfigure}{.32\textwidth}
  \centering
  \includegraphics[width=\linewidth]{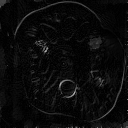}
  \caption{Error: max $2\cdot10^{-4}$ .} 
  \label{error lotus} 
\end{subfigure}
\vspace{.1cm}

\begin{subfigure}{.32\textwidth}
  \centering
  \includegraphics[width=\linewidth]{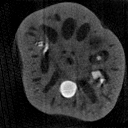}
  \caption{$L^2$-reconstruction $\lambda=0.1$.}
  \label{lotus_L2 1} 
\end{subfigure}
\begin{subfigure}{.32\textwidth}
  \centering
  \includegraphics[width=\linewidth]{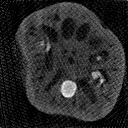}
  \caption{$L^2$-reconstruction $\lambda=0.01$.} 
  \label{lotus_L2 2} 
\end{subfigure}
\caption{Reconstructions for data obtained by a CT scanner, namely $12$ measurements with angles equally distributed in $[0,180]$.}
\label{fig:lotus}
\end{figure}
\section{Conclusions} \label{sec:conclusions}
In this paper, we extended the reconstruction model from \cite{LNOS19} to also cope with topology changes.
On the theoretical side, we were able to carry over all of the previous results.
Compared to \cite{GO18}, we utilize a simplified metamorphosis approach, which allows one to use non-convex regularizers at lower computational cost.
The chosen TV regularization enabled us to obtain satisfying reconstructions even for very limited data without suffering from the typical artifacts.
So far, our experiments are a proof-of-concept.
In the future, we also want to work with larger real data. 
To this end, it could be necessary to use more sophisticated (maybe even problem-tailored) regularization methods for $\mathbf z$ or different data terms $\mathcal D$, which can be incorporated into our model without much effort.
Again, we stress that the method can be easily extended to higher dimensions and other forward operators $\mathbf K$.
Even without the scope of real data, this seems to be a natural direction of future research as our method is designed in a modular way. 

\section*{Acknowledgment}
The authors want to thank Audrey Repetti for fruitful discussions on variable metric optimization methods.

\bibliographystyle{abbrv}
\bibliography{references_clean}
\end{document}